\documentclass[9pt]{amsart}
\usepackage[utf8]{inputenc}
\newcommand{\placeholder}{\quad}
\usepackage[T1]{fontenc}

\usepackage{enumitem}
\usepackage{amsmath,amsthm,amssymb}
\usepackage[french,british]{babel}
\usepackage{mathrsfs}
\usepackage[bookmarks=true,backref,colorlinks=true,citecolor=blue,urlcolor=blue,linkcolor=magenta]{hyperref}

\author
{Rodolphe~\textsc{Richard}}
\address{Rodolphe \textsc{Richard}\\
University College London,\\
Dept. of Mathematics,\\
25 Gordon Street\\
London\\
WC1H 0AY\\
United Kingdom
}
\email{r.richard@ucl.ac.uk}


\newtheorem{thm}{Theorem}[section]

\numberwithin{equation}{section}
\newtheorem{theorem}[thm]{Theorem}

\newtheorem{definition}[thm]{Definition}
\newtheorem{proposition}[thm]{Proposition}
\newtheorem{conjecture}[thm]{Conjecture}
\newtheorem{lemme}[thm]{Lemma}
\theoremstyle{definition}

\newtheoremstyle{explication}
  {\topsep}
  {\topsep}
  {\itshape}
  {1em}
  {\bfseries}
  {}
  {0pt}
  {}
\theoremstyle{explication}

\newcommand{\Nm}[1]{{\left\|{#1}\right\|}}

\newcommand{\Tr}{{\mathbf{Tr}}{}}
\newcommand{\Nmred}{{\mathbf{Nm}}}

\newcommand{\ol}[1]{\overline{#1}}

\newcommand{\A}{\mathbf{A}}
\newcommand{\Z}{\mathbf{Z}}
\newcommand{\Q}{\mathbf{Q}}
\newcommand{\F}{\mathbf{F}}

\newcommand{\C}{\mathbf{C}}

\newcommand{\eps}{\varepsilon}
\newcommand{\Ocal}{\mathcal{O}}

\newcommand{\tens}{\otimes}

\newcommand{\abs}[1]{{\left|{#1}\right|}}

\newcommand{\covol}{\mathrm{covol}}
\newcommand{\red}{\mathrm{r{e}d}}

\newcommand{\R}{\mathbf{R}}

\DeclareMathOperator{\End}{End}

\DeclareMathOperator{\Spec}{Spec}

\title{A two-dimensional arithmetic André-Oort problem}
\date{\today}
\begin{document}
\titlepage


\begin{abstract} 
We state and investigate an integral  analogue of the Andr\'e-Oort conjecture (in integral models 
of Shimura varieties). 
We establish an instance of this conjecture: the case of a modular curve, \emph{as a scheme over~$\mathbf{Z}$}.
It is a scheme of dimension two and, already in this case, our conjecture is highly non-trivial.
Our approach relies on equidistribution estimates related to subconvexity in analytic number theory and
our result is unconditional.
\end{abstract} 
\maketitle
\tableofcontents
\section{Introduction}
The André-Oort conjecture asserts that in a Shimura variety~$S$, the Zariski closure~$\ol{\Sigma}$ of any set of special points~$\Sigma\subseteq S(\ol{\Q})$ is a finite union of "special subvarieties". 
We refer to~\cite{AOtous} and its introduction for a statement and for recent advances on the André-Oort conjecture.

In dimension one cases, such as~$S=Y(1)$ (the modular curve), the André-Oort conjecture is trivial. 
The first non trivial cases occur in dimension two: for example the product~$S=Y(1)\times Y(1)$ of two modular curves.
This was proven by André~\cite{Andre}, unconditionally, and by Edixhoven~\cite{Edixhoven} too, by another method relying on the generalised Riemann hypothesis (GRH) for quadratic fields. 
Pila has treated this case in  \cite{Pila} using the Pila-Zannier o-minimal strategy that was a strating point of a tremendous recent
activity around the Andr\'e-Oort conjecture and its generalisations.

In this article we state the "horizontal part" of an "André-Oort conjecture in an arithmetic pencil" (see \S.\ref{sec:hAO}).
This extends the André-Oort conjecture, taking into account integral models of Shimura varieties. 
Our main result establishes the simplest case of our conjecture:
the case of~$S=Y(1)_{\Z}$ which is a two dimensional scheme.

This first case turns out to be highly non-trivial. We used advanced tools from analytic number theory 
which are not used in approaches
to the classical André-Oort conjecture~\cite{AOtous} and its generaisations.

There is an analogue of our conjecture for abelian varieties ("Manin-Mumford
conjecture in arithmetic pencils") which has been established in~\cite{BRU}.	

\subsection{}The present work was inspired by~\cite{Edixhoven Richard} in which analogues of the Andr\'e-Oort conjecture
in positive characteristic were considered.

In positive characteristic~$p$, the naive analogue of the André-Oort conjecture is trivially false. Every point in~$S(\ol{\F_p})$ is a special point, whereas (if~$S$ is of dimension at least two), most sub-varieties are not special. 

Still, a natural analogue in characteristic $p$ is given in~\cite{Edixhoven Richard}. That article studies the case~$S=Y(1)\times Y(1)$ over~$\F_p$. Rather than considering \emph{individual special points}, the authors consider a class of \emph{finite subsets of points}. 
 We refer to these finite subsets as ``special $0$-cycles'' in Def.~\ref{main def}. In~\cite{Edixhoven Richard} an analogue of the André-Oort conjecture established in the case of~$S=Y(1)\times Y(1)$ over~$\F_p$. The method is an adaptation of Edixhoven's method (~\cite{Edixhoven}), and the result is conditional on the Generalised Riemann Hypothesis for imaginary quadratic fields. 
 

\subsection{}
In this article, we consider the
 ``arithmetic'' setting (cf.~\S\ref{section une}): we allow the characteristic~$p$ to vary. 
We consider the case of~$S=Y(1)=\A^1$ over~$\Z$. The integral model is the scheme~$\A^1_\Z$, which is of dimension two: it is of relative dimension one (one geometric dimension), and the base~$\operatorname{Spec}(\Z)$ is of dimension one (one arithmetic dimension). 

Our main result establishes our conjecture in this~$S=Y(1)$ case. Our result is  unconditional.
\subsection{Summary of the article}\label{sec:summary} In Section~\ref{section une} we state precisely our two dimensional arithmetic problem and our main result. We
also state the ''integral André-Oort conjecture''. In section~\ref{section deux} we study some easy cases and identify the hard case. We then reduce the hard case to an equidistribution property. In section~\ref{section trois} we reduce this problem to proving some uniform upper bounds for some arithmetic multiplicities. In section~\ref{section quatre} we introduce quadratic forms, and complete the proof using appropriate estimates on representation numbers of some ternary quadratic forms. One of these estimates is due to Duke, relying on work of Iwaniec, Duke and Schulze-Pillot, and is now known as "sub-convexity" in analytic number theory. Another estimate is needed and it is established in the last section~\ref{section last}.

An appendix recalls some invariants of Gross lattices attached to some Brandt algebras, and extends a result of Duke and Schulze-Pillot for the case under study.

\subsection{Acknowledgements} This work was carried out at the University of Cambridge and at University College London. The author thanks both institutions for hospitality.
 The author was supported by the Leverhulme Trust grant RPG-2019-180 which is gratefully acknowledged. The author thanks Andrei Yafaev for numerous comments.

\section{Statement of the problem and Main result}\label{section une}
The following notations are standard. \emph{We denote by~$\overline{\Q}$ an algebraic and algebraically closed extension of~$\Q$. For each prime~$p$ we denote by~$\F_p$ a prime field of characteristic~$p$ and~$\overline{\F_p}$ an algebraic and algebraically closed extension thereof, and~$\F_{p^2}$ the quadratic extension of~$\F_p$ in~$\overline{\F_p}$.} 

The underlying subsets of the~$\overline{\F_p}$, for varying primes, are chosen to be disjoint.\footnote{For Definition~\eqref{Definition}.}
We also denote by~$\overline{\Z}\leq \overline{\Q}$ the subring of algebraic \emph{integers}.

The special points of~$Y(1)(\ol{\Q})\simeq \ol{\Q}$ are the modular invariants of elliptic~$E$ curves with~$CM$: their endomorphism ring~$\End(E)$ is isomorphic to an order in a quadratic imaginary extension of~$\Q$. These are classically known as \emph{singular invariants}. Every singular invariant has an associated "discriminant"~$\Delta$: the discriminant of the quadratic order~$\End(E)$. We will call \emph{negative discriminant number} the numbers obtained this way: equivalently~$\Delta\in\{0;1\}+4\cdot\Z_{<0}$ (cf.~\cite[p.\,52]{Kowalski}). After Weber and Fueter,
(see~\cite[\S C.11, Th. 11.2]{Silverman},) 
we know that every singular invariant is an algebraic integer (it belongs to~$\ol{\Z}$), and that two singular invariants share the same discriminant if and only if they are algebraically conjugated: they belong to the same~$\operatorname{Gal}(\ol{\Q}/\Q)$-orbit. In the scheme-theoretic language, this Galois orbit corresponds to a unique closed point~$z_{\Delta}$ in the scheme~$\A^1_{\Q}$.

Here is our main definition. We refer to~\cite{Edixhoven Richard} for a similar setting.
\begin{definition}\label{main def}
 For any~$\Delta\in\{0;1\}+4\cdot\Z_{<0}$, we denote by~$S(\Delta)\subseteq\overline{\Z}$ the set of singular invariants with associated discriminant~$\Delta$. A \emph{special~$0$-cycle} is a set of the form (for some prime~$p$ and some~$\Delta\in\{0;1\}+4\cdot\Z_{<0}$),
\[
s=\red_p(S(\Delta))\subseteq Y(1)(\ol{\F_p})=\ol{\F_p}
\]
where~$\red_p:\overline{\Z}\to\overline{\F_p}$ is any ring homomorphism. 

The Zarsiki closure~$\overline{\{z_\Delta\}}$, in the scheme~$\A^1_\Z$, is called \emph{a singular section}
\emph{(of discriminant~$\Delta$)}.
\end{definition}
\noindent Remark. The scheme~$Z=\overline{\{z_\Delta\}}$ satisfies~$Z(\ol{\Q})=S(\Delta)$ and~$Z(\ol{\F_p})=\red_p(S(\Delta))$. 
In particular, the set~$\red_p(S(\Delta))$ does not depend on the choice of a "reduction map"~$\red_p:\ol{\Z}\to\ol{\F_p}$.

The problem we are interested in is: what is
\begin{equation}\label{problem}
\text{the Zariski closure in the scheme~$\A^1_\Z$ of a family~$(s_i)_{i\in I}$ of special $0$-cycles~$s_i=\red_{p_i}(S(\Delta_i))$.}
\end{equation}
 
Obvious closed subsets of the scheme~$\A^1_\Z$ arise this way, such as:
\begin{enumerate}
\item \label{cas a}any special $0$-cycle~$\red_p(S(\Delta))$ itself: when~$p_i=p$ and~$\Delta_i=\Delta$ are fixed;
\item \label{cas b}singular sections~$\overline{\{z_\Delta\}}$: when~$p_i$ ranges through primes and~$\Delta_i=\Delta$ is fixed;
\item \label{cas c}fibres~$\A^1_{\F_p}$ over a prime~$p$: when~$p_i=p$ is fixed and~$\Delta_i$ ranges through~$\{0;1\}+4\cdot\Z_{<0}$;
\item \label{cas d}the ambient space~$\A^1_\Z$ itself: for the family of all special~$0$-cycles.
\end{enumerate}


Our main result is a converse statement.
\begin{theorem}\label{main thm}
 Let~$\Sigma$ be a set of special $0$-cycles in the sense of Def.~\ref{main def}, and denote by~$\overline{\bigcup\Sigma}$ the Zariski closure in the scheme~$\A^1_\Z$ of~$\bigcup\Sigma=\bigcup\left\{s \mid s\in \Sigma\right\}$.

Then~$\overline{\bigcup\Sigma}$ is finite union of subsets, each of the form~\eqref{cas a},~\eqref{cas b},~\eqref{cas c}, or~\eqref{cas d}.
\end{theorem}

\subsection{Horizontal part of André-Oort conjecture in arithmetic pencils}\label{sec:hAO}
We record for future reference a case of
an André-Oort conjecture in arithmetic pencils.\footnote{The "horizontal" case
refers to the condition~$\#O_E/\mathfrak{p}_n\to +\infty$. We defer the discussion of "vertical part" of the conjecture.} We use the scheme-theoretic language.

Let~$S$ be a Shimura variety over a number field~$E$ denote by~$O_E\leq E$ the subring of integers, and let $\mathscr{S}$ be a model of~$S$ over~$O_E[1/N]$ for some~$N\in E^\times$. For every special subvariety~$Z\subseteq S$, denote by~$\ol{Z}$ its Zariski closure\footnote{Our convention is that the Zariski closure of a variety~$Z\subseteq S$ defined over an extension~$B/E$ is the closure of the image of the scheme~$Z$ in the scheme~$\mathscr{S}$.} in the scheme~$\mathscr{S}$. For every prime~$\mathfrak{p}$ of~$O_E[1/N]$, we denote by~$\mathscr{S}_{\mathfrak{p}}=\mathscr{S}\tens_{O_E}O_E/\mathfrak{p}$ the fibre of~$\mathscr{S}$ above~$\mathfrak{p}$ and define
\[
\ol{Z}_{\mathfrak{p}}:=\ol{Z}\cap \mathcal{S}_{\mathfrak{p}}.
\]
We conjecture the following. 
The analogue conjecture in the context of Manin-Mumford has been obtained in~\cite{BRU}. 
\begin{conjecture}[Horizontal part of André-Oort conjecture in arithmetic pencils]\label{hAO}
 let~$z_n\in S(\ol{\Q})$ be a sequence of special points, and~$\mathfrak{p}_n$ a sequence of primes of~$O_E[1/N]$,
and define~$E_n:=(\ol{Z_n})_{\mathfrak{p}_n}$ where~$Z_n$ is the singleton~${\{z_n\}}$. Assume~$\#O_E/\mathfrak{p}_n\to +\infty$.
Then the Zariski closure~$V:=\ol{\bigcup_{n\in\Z_{\geq0}} E_n}$ can be written as
\[
V=E_{n_1}\cup\ldots\cup E_{n_i}\cup \ol{Z'_1}\cup\ldots \cup \ol{Z'_j},
\]
with~$n_1,\ldots,n_i\in\Z_{\geq0}$ and where~$Z'_1\cup\ldots \cup Z'_j\subseteq S$ are special subvarieties.
\end{conjecture}
\noindent Remark. This "horizontal part" Conjecture~\ref{hAO} does not depend on the choice of the model~$\mathscr{S}$. It implies the classical André-Oort conjecture for~$S$. The current successful approach to the classical André-Oort conjecture is recalled in~\cite{AOtous}. Many of the tools used in the classical case
can't be expected to make sense in the extension above. Indeed, even in the simplest case~$S=Y(1)$, studied in this article, we encountered new difficulties, and needed to introduce different tools.


\section{Easy cases and the hard case}\label{section deux}

The global structure of the proof is an argumentation by cases, all but one being easy. Note that by the nature of the main statement, we can always partition~$\Sigma$ into finitely many subsets and treat each of these subsets separately.
\subsection{} We use the following terminology in order to delimit subcases.
\begin{definition}\label{Definition} The \emph{characteristic} of a special~$0$-cycle~$s$ is the\footnote{Special~$0$-cycles are non empty. The fields~$\overline{\F_p}$ were assumed disjoints.} prime number~$p$ such that~$s\subseteq\overline{\F_p} = \A^1(\overline{\F_p})$.

The \emph{degree} of a special~$0$-cycle~$s$ is its cardinality.

A special~$0$-cycle~$s$ either consists of only ordinary invariants, or of only supersingular invariants. We call it \emph{ordinary} or \emph{supersingular} accordingly. 
\end{definition}

Observe that the degree of a special $0$-cycle~$\red_p(S(\Delta))$ is bounded by~$h(\Delta)=\abs{S(\Delta))}$, hence finite.

\subsection{}Our first reduction steps are achieved by treating the easy cases.
\subsubsection{Bounded characteristic}
The case of bounded characteristic is easy. This is a problem of Krull dimension one, and it is trivial. (By decomposing~$\Sigma$ further we may even assume the characteristic is fixed). This can also be deduced from~\cite{Edixhoven Richard}, which is more general but conditional.
\subsubsection{Divergent subsequence}
The case where~$\Sigma$ contains an infinite subsequence with divergent characteristic and divergent degree is easy. Let us prove that the closure~$\overline{\bigcup\Sigma}$ is then the ambient space~$\A^1_\Z$.
\begin{proof}
Let~$Z$ be the union of the irreducible components of~$\overline{\bigcup\Sigma}$ which are of prime characteristic: contained in some fibre~$\A^1_{\F_p}$. As the characteristic is divergent along our subsequence, only finitely many special~$0$-cycles from this  sequence are contained in~$Z$. We cast away these special~$0$-cycles, as we may. It follows the Zariski closure~$\overline{\bigcup\Sigma}$ is ``of characteristic~$0$'', that is\footnote{The generic points of its irreducible components are sent to the generic point of~$\Spec(\Z)$.} ``flat over~$\Z$''. Hence~$\overline{\bigcup\Sigma}$ is a finite section\footnote{The closure of finitely many points of~$\A^1_\Q$, the sum of degrees of which is~$d$.} (of relative dimension~$0$) of some generic degree~$d$, or the ambient space. The latter case agrees with the claim. It is enough to see the former case cannot happen. Indeed the fibres of~$\overline{\bigcup\Sigma}$ over a prime~$p$ have at most~$d$ points, hence can only support cycles special~$0$-cycles of bounded degree. This contradicts an assumption: the divergence of the degree along our subsequence.
\end{proof}
\subsubsection{Reduction step}

We may assume that~$\Sigma$ does not contain an infinite subsequence with divergent characteristic and divergent degree. It follows that~$\Sigma$ can be partitioned into two families, one of bounded characteristic, one of bounded degree. We may consider each subfamily separately. We already treated the first case. We may henceforth assume that elements in~$\Sigma$ have bounded degree	; and moreover assume that the characteristic is divergent: there are only finitely many special~$0$-cycles of bounded degree in each characteristic.

We reduced the problem to the case of a sequence of special~$0$-cycles of bounded degree, and divergent characteristic.

As in the proof above, we may furthermore assume that~$\overline{\bigcup\Sigma}$ is of characteristic~$0$: a finite section or the ambient space.

\subsubsection{Ordinary case}

We partition~$\Sigma$ into ordinary and supersingular special~$0$-cycles and treat each case independently. 
%
%
%

The case of ordinary cycles is not hard. We use that they are unramified in the following sense.
\begin{definition}
A special~$0$-cycle~$s$ is said to be \emph{unramified} if it is the fibre of~$\overline{\{z_\Delta\}}$ over a prime~$p$ for a singular invariant~$z_\Delta$ which is unramified at~$p$; equivalently:~$\abs{\red_p(S(\Delta))}=\abs{S(\Delta)}$. 
Then the degree of~$s$ is the degree of~$z_\Delta$.

We call such a~$z_\Delta$ an \emph{unramified lift} of~$s$.

For an ordinary special~$0$-cycle there is a unique such~$z_\Delta$, given by the canonical lift. (See also Deuring's theorem~\cite[\S13.4 Th\ 12\,(ii), Th.\ 13]{Lang Elliptic} from~\cite{Deuring}.)
\end{definition}
We also use the following fact.
\paragraph{Fact}\label{Fact h Delta} \emph{There are finitely many singular sections of bounded degree.} Equivalently the degree~$[\Q(z_\Delta):\Q]$ diverges with the discriminant. This degree is a class number~$h(\Delta)$ by the theory of complex multiplication. The divergence of~$h(\Delta)$ is a conjecture of Gau\ss{} solved by Heilbronn \cite[\S22.4]{Kowalski}. By a more precise result~\cite{Siegel} of Siegel, the degree follows the asymptotic~
\begin{equation}\label{Siegel estimate}
h(\Delta)=\Delta^{1/2+o(1)}.
\end{equation}

Unramified special~$0$-cycles of bounded degree at most, say,~$D$, have each an unramified lift of degree at most~$D$, and an unramified lift will have bounded discriminant at most~$\Delta(D)\approx D^{2+o(1)}$.  We will write our proof in terms of a slightly weaker property.

\begin{definition} A family of special~$0$-cycles is said to have property, say, P if any subsequence of divergent characteristic and bounded degree arise as reduction from special sections of bounded discriminant.
\end{definition}
In a formula, for such a sequence~$(s_i)_{i\geq 0}$
\[
\forall D>0, \exists \Delta(D), \forall i \left[\# s_i\leq D\implies \left(\exists \Delta\leq\Delta(D), \exists p, s_i=\red_p(S(\Delta))\right)\right].
\]

We now prove the theorem for a family~$\Sigma$ with property~P.
\begin{proof}We already reduced the theorem to the case of a subsequence of bounded degree and divergent characteristic.

We may apply property P to our sequence. Our special $0$-cycles arise as reduction from special sections of bounded discriminant. We are left with finitely many discriminants, and everything occurs in a finite union of singular sections. This is again a problem of Krull dimension one, and is again trivial. (We can decompose further and assume the discriminant is fixed, in which case we are in a single, irreducible, singular section).
\end{proof}
\subsection{The hard case}
We have identified a last case, the case of supersingular special~$0\text{-}$cy\-cles of bounded degree and unbounded characteristic.  We complete the proof of~\ref{main thm} provided we can prove the following. The remaining of this article is devoted to the treatment of this case. 
\begin{theorem}\label{theorem P}
The set of supersingular special~$0$-cycles satisfies property P. Namely, for every upper bound~$C<\infty$, there exist~$p(C)<\infty$ and~$\Delta(C)<\infty$ such that the following holds.

For every special~$0$-cycle~$s$ of degree at most~$C$, its characteristic~$p$ is at most~$p(C)$, or otherwise there is a discriminant~$\Delta$ at most~$\Delta(C)$, such that~$s$ is the reduction of~$S(\Delta)$ in characteristic~$p$. 
\end{theorem}

We do not know an easy proof. This article relies on elaborate results to prove this last statement.
\subsubsection{}

Let us point out that the closely related statement 
\begin{equation}\label{limite naive}
\lim_{(\Delta,p)\to(-\infty,+\infty)} \#\red_p(S(\Delta))\stackrel{?}{=}+\infty
\end{equation}
does not hold, even in the case of ordinary reduction. This is because of the invariance phenomenon\footnote{Both members are in the orbit of the other for the Hecke operator~$T_p$, as can be checked before reduction over~$\C$ in terms of ideals, and then seen locally, id\'{e}lically. That is, by Eichler-Shimura relation, they have the same orbit under Frobenius. But these sets are already definable over~$\F_p$.}
\begin{equation}\label{phenomenon}
\red_p(S(\Delta))=\red_p(S(\Delta\cdot p^2)).
\end{equation}
A counter example would be~$\Delta=-3p^2$ with~$\red_p(S(\Delta))=\red_p(S(-3))=\{0\pmod{p}\}$.
\subsubsection{}

We introduce some definitions in order to discuss more precisely the case we have identified. We begin with standard definitions, some of which we already used.
\begin{definition}A \emph{discriminant number} is an integer of the form~$\Delta=B^2-4AC$, that is with quadratic residue modulo~$4$. Its \emph{conductor} is the biggest integer~$f(\Delta)>0$ such that~$\Delta_{\text{f.}}=\Delta/f(\Delta)^2$ is a discriminant. A discriminant of the form~$\Delta_{\text{f.}}$, that is of conductor~$1$, is called \emph{fundamental}.

In terms of quadratic orders,~$\Delta$ is a discriminant if and only if~$\tau_\Delta=\frac{\Delta+\sqrt{\Delta}}{2}$ is an algebraic integer, that is~$\Ocal_\Delta=\Z[\tau_\Delta]$ is an imaginary quadratic order. We characterize~$f(\Delta)$ by~$\Ocal_\Delta=\Z+f(\Delta)\cdot\Ocal_{\Q(\sqrt{\Delta})}$ and~$\Delta_{\text{f.}}$ by~$\Ocal_{\Q(\sqrt{\Delta})}=\Ocal_{\Delta_{\text{f.}}}$.

We denote by~$h(\Delta)=\abs{S(\Delta)}=\abs{Pic(\Ocal_\Delta)}$ the class number of~$\Ocal_\Delta$.
\end{definition}
We introduce the following definition.
\begin{definition}
Let~$p$ be a prime.

We say that~$\Delta$ is \emph{$p$-fundamental} if its conductor is prime to~$p$. 

We denote by~$\Delta_{\text{$p$-f.}}=\Delta_{\text{f.}}\cdot {f}^2$ where~$f$ is the prime-to-$p$ part of~$f(\Delta)$, the most negative $p$-fundamental discriminant number dividing~$\Delta$.
\end{definition}
\begin{proposition}
A special $0$-cycle is the reduction of~$S(\Delta)$ for some $p$-fundamental~$\Delta$.
\end{proposition}
\begin{proof}Iterating equality~\eqref{phenomenon} to reduce the discriminant, we end up by Fermat descent with 
\[
\red_p(S(\Delta))=\red_p(S(\Delta_{\text{$p$-f.}})).\qedhere
\]
\end{proof}
\subsubsection{Limit formula}

We will prove the following. This is a strengthening of theorem~\ref{theorem P}, as we prove that we can take as~$\Delta$ any $p$-fundamental discriminant. This shows that phenomenon~\eqref{phenomenon} is the ultimate obstacle to the validity of~\eqref{limite naive}.
\begin{theorem} For every upper bound~$C<\infty$, there exist~$p(C)<\infty$ and~$\Delta(C)<\infty$ such that the following holds.
For every prime~$p>p(C)$ and $p$-fundamental discriminant~$\Delta>\Delta(C)$, the reduction of~$S(\Delta)$ in characteristic~$p$ has more than~$C$ elements.

Equivalently
\begin{equation}\label{limit theorem}
\lim_{\substack{(\Delta,p)\to(-\infty,+\infty) \\p\nmid f(\Delta)}} \#\red_p(S(\Delta)){=}+\infty.
\end{equation}
\end{theorem}
The divergence of~$\Delta$ is obviously necessary. 
The divergence of~$p$ on top of that of~$\Delta$ is necessary for supersingular special~$0$-cycles: for a fixed prime~$p$
\[
\lim_{\substack{\Delta\to-\infty\\ p\nmid f(\Delta)\\ \red_p(S(\Delta))\subseteq SS(p)}} \#\red_p(S(\Delta)){=}\#SS(p)=\frac{p-1}{12}+O_{p\to\infty}(1),
\]
where~$SS(p)$ denotes the supersingular locus. This follows for instance from effective sparse equidistribution result~\cite[Th.\,3]{Michel}. 

\section{Upper bounds Strategy}\label{section trois}
For each prime~$p$, choose a reduction map~$\red_p:\overline{\Z}\to\overline{\F_p}$. 
\subsection{Setting}

We denote by
\[SS(p)=SS(p)(\overline{\F_p})\subseteq{\F_{p^2}}\subseteq\overline{\F_p}
\qquad \text{and} \qquad S(\Delta)=S(\Delta)(\overline{\Z})\subseteq\overline{\Z}
\]
the set of supersingular invariants of characteristic~$p$, and the set of singular invariants of discriminant~$\Delta$. We are interested in lower bounds for the cardinality
\[
\abs{\red_p(S(\Delta))}
\]
where we can assume that~$\Delta$ is $p$-fundamental.
\subsection{}
We denote by~$[S(\Delta)]$ the standard cycle on the finite set~$S(\Delta)$, seen as the counting measure, or as the unity density function. We consider its direct image
\[
{\red_p}_\star[S(\Delta)]
\]
on~$\overline{\F_p}$. It is a finite measure whose support is the image~$\red_p(S(\Delta))\subseteq \overline{\F_p}$, and whose density~${\red_p}_\star[S(\Delta)](\{\bar\jmath\})$ at a point~$\bar\jmath\in\overline{\F_p}$ is the cardinality~$\#\{j\in S(\Delta)|\red_p(j)=\bar\jmath\}$ of the inverse image of~$\bar\jmath$ by the restricted reduction map~$\red_p|_{S(\Delta)}$.
\subsection{} Using theses multiplicities, we have the obvious bound
\begin{equation}\label{comparaison}
h(\Delta)=\sum_{\bar\jmath\in SS(p)}{\red_p}_\star[S(\Delta)](\{\bar\jmath\})\leq
\abs{\red_p(S(\Delta))}\cdot\max_{\bar\jmath\in SS(p)}{\red_p}_\star[S(\Delta)](\bar\jmath),
\end{equation}
where~$h(\Delta)=\abs{S(\Delta)}$ denotes the class number. Roughly speaking, our strategy to ensure that~$S(\Delta)$ visits many places of~$\overline{\F_p}$ is to check that it does not stay too long in each place. \emph{Our problem is reduced to obtaining upper bounds for these multiplicities}~${\red_p}_\star[S(\Delta)](\bar\jmath)$.

\subsection{} Precisely: By virtue of~\eqref{comparaison}, one reduces~\eqref{limit theorem} to the following.
\begin{proposition}With the above notations, the following equivalent statements hold
\[\max_{\bar\jmath\in \overline{\F_p}}{\red_p}_\star[S(\Delta)](\bar\jmath)=o_{\small \substack{\\(\Delta,p)\to(-\infty,+\infty) \\ p\nmid f(\Delta)}}(h(\Delta)),\]
\[\lim_{\substack{\\(\Delta,p)\to(-\infty,+\infty) \\ p\nmid f(\Delta)}}\frac{\max_{\bar\jmath\in \overline{\F_p}}{\red_p}_\star[S(\Delta)](\bar\jmath)}{h(\Delta)}=0.\]
\end{proposition}
\subsection{}
Note that in the ordinary case, the $p$-fundamental discriminants are precisely those giving rise to unramified special~$0$-cycles. We have then 
\[\max_{\bar\jmath\in \overline{\F_p}}{\red_p}_\star[S(\Delta)](\bar\jmath)=1=o(h(\Delta)).\]
\emph{We will restrict our attention henceforth to supersingular special~$0$-cycles.}

\section{Approach using Quadratic forms}\label{section quatre}
 For every supersingular invariant~$\bar\jmath$ there is a ``Gross lattice''~$S_{\bar\jmath}$, which is a Euclidean lattice well defined up to isomorphism\footnote{Let~$E$ be an elliptic curve over~$\overline{\F_p}$ with modular invariant~$j(E)=\bar\jmath$. Then~$R=\End(E)$ is an order in a quaternion algebra, and we denote by~$R^0$ the sub-lattice of pure quaternions in~$R$. We have a reduced norm form, and~$S_{\bar\jmath}$ is obtained by its restriction to~$S\cap R^0$ where~$S$ is the order~$S=\Z+2R$.}. It is
   of rank~$3$, with co-volume~$\covol(S_{\bar\jmath})=4\sqrt{2}p$ (Hessian determinant~$D(p)=2(4p)^2$) and level~$4p$. 
 
 It has furthermore the property that for every primitive representation of a positive integer~$-\Delta$ one can attach an element~$\tilde\jmath$ in~$S(\Delta_{\text{$p$-f.}})$ with~$\red_p(\tilde\jmath)=\bar\jmath$ (Deuring's lift), and that every~$\tilde\jmath$ arises in this way (Deuring's reduction of endomorphisms). In particular one has the bound
\[
{\red_p}_\star[S(\Delta_{\text{$p$-f.}})](\bar\jmath)
\leq 
r'(\abs{\Delta},S_{\bar\jmath})
\leq 
r(\abs{\Delta},S_{\bar\jmath})
\]  
where~$r'(\abs{\Delta},S_{\bar\jmath})$, resp.~$r(\abs{\Delta},S_{\bar\jmath})$ is the number of \emph{primitive} representations, resp. of representations, of the integer~$\abs{\Delta}$ by the euclidean lattice~$S_{\bar\jmath}$.

We refer to~\cite{Gross} for a detailed treatment of this. This owes notably to the work of Brandt, Deuring, Eichler, and Gross. See also~\cite[Lem.\,3.2]{Elkies}, which considers solely fundamental discriminants. 

For our purposes, it is enough to establish a uniform bound
\begin{equation}\label{eq2}
r'(\abs{\Delta},S_{\bar\jmath})=o(h(\Delta))\text{, or the even stronger one }r(\abs{\Delta},S_{\bar\jmath})=o(h(\Delta))
\end{equation}
as~$(\Delta_{\text{$p$-f.}},p)\to(-\infty,+\infty)$.

We will consider two cases, according to whether~$\Delta$ or~$p$ is large compared to the other. 
\subsection{Dirichlet-Hermite bound}
A first result is not a deep one. We postpone its proof until the next section. We establish in proposition~\ref{Hermite Dirichlet bound} that there is some constant~$c_1<+\infty$ and exponent~$\kappa>0$ such that, for~$S\to\Z$ any positive definite integral ternary quadratic form of co-volume~$q$,
\[
r(n,S)\leq c_1\cdot \left(n^{o(1)}+ n^{\frac{1}{2}}q^{-\kappa}\cdot n^{o(1)}\cdot q^{o(1)}\right),
\]
in which the~$o(1)$ notation abbreviates explicit arithmetic functions and do not depend on~$S$.

Applied to~$S=S_{\bar\jmath}$, together with Siegel estimate, this gives us
\[
r(\abs{\Delta},S_{\bar\jmath})/h(\Delta)=O(1)\cdot\left(\abs{\Delta}^{-\frac{1}{2}+o(1)}+\abs{\Delta}^{o(1)}\cdot p^{-\kappa+o(1)}\right).
\]
Here the exponent~$o(1)$ of~$\Delta$ still does not depend on~$S$ but is ineffective: it relies on Siegel's theorem.

The asymptotic of the first term agrees with~\eqref{eq2}. As for the second term it agrees with~\eqref{eq2}, for any~$\eps>0$, provided
\[
\Delta=O(p^{\kappa/\eps})=o(p^{\kappa/o(1)}).
\]
which includes the range
\begin{equation}\label{range 1}
\log(p)\geq\frac{\eps}{\kappa}\log\abs{\Delta}.
\end{equation}
\subsection{A 	Conditional bound} Let us mention the work~\cite{Kane,KanePhD}, which produces an effective equidistribution bound,  \emph{assuming the GRH L-series of Dirichlet and of modular forms of weight~$2$}. Unfortunately, it is only proved \emph{for fundamental discriminants}.

By~\cite[Th. 1.9]{Kane}, there is a constant~$C(\eps)$ such that for all primes~$p$ and all~$\Delta\geq C(\eps)p^{14+\eps}$ a fundamental and $p$-supersingular discriminant,
\[
\#\red_p(S(\Delta))=\#SS(p)=\frac{p-1}{12}+O(1).
\]
This establishes \eqref{limit theorem} in the range
\[
\log(p)\leq \frac{1}{14+\eps}(\log(\Delta)-\log(C(\eps)).
\]
This suffices to, conditionally, establish our results, for fundamental discriminants.
 
The approach of the above-mentioned work is similar to the arguments below.

\subsection{Subconvexity bound}\label{section Duke}
As for the second result, it will involve both upper and lower bounds on representations numbers~$r(\abs{\Delta},S_{\bar\jmath})$,
though we ultimately only want upper bounds on \emph{primitive} representation numbers~$r'(\abs{\Delta},S_{\bar\jmath})$.

We first note that there is the following relation (we recall that~$\Delta\neq0$)
\[
r(\abs{\Delta},S_{\bar\jmath})=\sum_{\left.f^2\middle| \abs{\Delta}\right.}r'(\abs{\Delta}/f^2,S_{\bar\jmath}),
\]
to which we can apply the Möbius inversion
\begin{equation}\label{Moebius}
r'(\abs{\Delta},S_{\bar\jmath})=\sum_{\left.f^2\middle| \abs{\Delta}\right.}\mu(f) r(\abs{\Delta}/f^2,S_{\bar\jmath}),
\end{equation}
in which the sum has at most~
\begin{equation}\label{terms sigma bound}
\sigma_0(f(\Delta))=\abs{\Delta}^{o(1)}
\end{equation}
 terms (see \cite[\S13.10 p\,296]{Apostol} or \cite[p.6, footnote]{RichardJTNB}).
\subsubsection{}
We first use a result of~\cite{Duke ternary}: with~$N=11/2$ and~$\gamma=1/28$, and for \emph{square-free~$\Delta$}
\begin{equation}\label{Duke ternary}
r(\abs{\Delta},S_{\bar\jmath})=r(\abs{\Delta},gen(S_{\bar\jmath}))\cdot \frac{1}{M}+O(1)\cdot D(p)^{N+o(1)}\cdot \abs{\Delta}^{1/2-\gamma+o(1)}.
\end{equation}
Here~$r(\abs{\Delta},gen(S_{\bar\jmath}))$ is the number of all the representations of~$\abs{\Delta}$ by the representatives~$Q$ of genus of~$S_{\bar\jmath}$, weighted by the inverse of the number~$\abs{O(Q)}$ of automorphs, which only depends on~$p$. The coefficient~$M$, is the total  \emph{ma\ss} (measure, weight) of the genus (which seems to be~$M=(p-1)/48$, or~$(p-1)/24$ for the weights~$1/\abs{SO(Q)}$).
\paragraph{}This is again valid for any integral ternary quadratic form instead of~$S_{\bar\jmath}$. The bound is uniform with respect to the euclidean lattice~$S_{\bar\jmath}$, hence uniform in~$p$ in our setting.
\paragraph{}
Square-free integers cover odd fundamental discriminants. This is extended to all discriminant numbers with~\eqref{final bound}.

\subsubsection{}We apply Möbius inversion~\eqref{Moebius} to~\eqref{Duke ternary}. With \eqref{terms sigma bound} it yields
\begin{equation}\label{Duke ternary 2}
r'(\abs{\Delta},S_{\bar\jmath})=r'(\abs{\Delta},gen(S_{\bar\jmath}))\cdot \frac{1}{M}+O(1)\cdot D(p)^{N+o(1)}\cdot \abs{\Delta}^{1/2-\gamma+o(1)}.
\end{equation}
where~$r'(\abs{\Delta},gen(S_{\bar\jmath}))$ is the weighted sum of primitive representations.

We will see in \S\ref{Gross computation}
\begin{equation}\label{eq1}
r'(\abs{\Delta},gen(S_{\bar\jmath}))/M\leq 24\cdot h(\Delta)/({p-1})=o(h(\Delta))\text{ as }p\to\infty.
\end{equation}
The first term of~\eqref{Duke ternary 2}	which agrees with~\eqref{eq2}. The second term agrees with~\eqref{eq2} in the range
\[
p^2=O(\Delta^{\frac{\gamma}{N}-\eps'})=o(\Delta^{\frac{\gamma}{N}-o(1)}).
\]
which contains the domain
\begin{equation}\label{range 2}
\log(p)\leq\frac{1}{2}\left({\frac{\gamma}{N}-\eps'}\right)\log\abs{\Delta}.
\end{equation}
\subsection{}
By choosing~$\eps$ and~$\eps'$ so that~$\eps/\kappa<(\gamma/N-\eps')/2$, which is possible, we cover all cases with our both ranges~\eqref{range 1} and \eqref{range 2}. This ends the proofs.

\subsection{}\label{Gross computation}
 As remarked\footnote{We believe that only the equality of the ratios
\[
\frac{1}{M}\theta(gen(S_{\bar\jmath});q)=\frac{12}{p-1}G(q)
\]
does holds, not that of the numerators and denominators as implied in the reference~\cite[(3.3)]{Elkies}. The important part for us, their argumentation using the canonicity of the Eisenstein/cuspidal decomposition, remains unharmed.

We believe proportionality factor~$4$ is needed: compare~$(2-1)/12$ with the number~$48$ of automorphs of the corresponding Gross lattice. A factor~$2$ at least if one restrict to the~$24$ automorphs in~$SO(Q)$. 

The genus theta series is a sum over the type number of a quaternion algebra, and Gross series is a sum over the class number, roughly twice as many elements. It seems that supersingular elliptic curve with modular invariant over the prime field have twice as many automorphs in the corresponding group~$SO(Q)$ than automorphisms coming from unit of the quaternion order. A likely candidate for the extra automorph is conjugation by the Frobenius endomorphism.} in~\cite{Elkies}, we can equate the series~$\theta(gen(S_{\bar\jmath});q)=\sum r(\abs{\Delta},gen(S_{\bar\jmath}))\cdot q^{\abs{\Delta}}/M$  with a series~$\frac{12}{p-1}G(q)$ computed by Gross. We obtain, after translating from representations to primitive representations
\[
\frac{r'(\abs{\Delta_{\text{$p$-f.}}},gen(S_{\bar\jmath}))}{M}=\eps\cdot\frac{12}{p-1}\cdot h(\Delta_{\text{$p$-f.}})/u(\Delta_{\text{$p$-f.}})\text{ with }\eps
=
\begin{cases}
0\text{ if $\Delta$ is ordinary at~$p$,}\\
1\text{ if $\Delta$ once divisible by~$p$,}\\
2\text{ otherwise,}
\end{cases}
\]
for $p$-fundamental discriminants, where~$u(\Delta)=\abs{{\Ocal_{\Delta}}^\times/\Z^\times}\in\{1;2;3\}$. For discriminant divisible by~$p^2$ we have~$r'(\abs{\Delta_{\text{$p$-f.}}},gen(S_{\bar\jmath}))=0$. We conclude
\[
\frac{r'(\abs{\Delta_{\text{$p$-f.}}},gen(S_{\bar\jmath}))}{M}\leq 24\cdot h(\Delta_{\text{$p$-f.}})/(p-1)\leq24\cdot h(\Delta)/(p-1).
\]

\section{High characteristic Bound}\label{section last}
\subsection{The Bound}
Let~$Q$ be a positive definite, ternary, integral, quadratic form: that is we can write~\[Q=\sum_{1\leq i\leq j\leq 3} a_{i,j}x_ix_j\]
with integer coefficients~$(a_{i,j})_{1\leq i\leq j\leq 3}$, and the Hessian matrix~$A=(\partial_i\partial_j Q)_{1\leq i,j\leq 3}$, which is symmetric with integral entries and even diagonal, 
 is positive definite. The Hessian determinant of~$Q$ is denoted~$H(Q)=\det(A)$. The co-volume~$\covol(Q)$ of~$Q$ is~$\sqrt{H(Q)/8}$.

We wish to bound the number~$r(n,Q)$ of representations of an integer~$n$ by~$Q$, that is, of integer solutions to~$Q=n$.
To do this, we introduce the following variant
\[
\widetilde{\sigma_0}(n)=\max_{1\leq m\leq n}\sigma_0(m)=n^{o(1)}
\]
of the number of divisors function~$\sigma_0(n)=\sum_{0<d\mid n}d^0$. 
We find at~\cite[\S13.10 p.\,296]{Apostol} that
\[
\sigma_0(n)=\sum_{0<d\mid n}d^0\leq (\log(2)+o(1))\cdot\frac{\log(n)}{\log\log(n)},
\]
hence we have a the sub-polynomial growth behaviour 
\[
\widetilde{\sigma_0}(n)=n^{o(1)}.
\] 
\begin{proposition}\label{Hermite Dirichlet bound}
There is an exponent~$\kappa=1/6>0$, and a constant~$C=6$, such that for every positive definite integral ternary quadratic form~$Q$ of Hessian determinant~$q$,
\[
\forall n>0,
r(n,Q)\leq C\cdot\left(\sigma_0(n)+2{\sqrt{n}} \cdot \frac{1
}{(q/2)^{\kappa}}\widetilde{\sigma_0}((2q)^{2/3} n)\right)=n^{o(1)}+\sqrt{n}\cdot q^{-\kappa}\cdot (qn)^{o(1)}.
\]
\end{proposition}
An important point for us is the negative polynomial dependency on the determinant. 


\subsection{The Slices method}
For a sub-lattice~$\Pi$ of rank~$2$, 
let~$R=Q\restriction_\Pi$ be the restriction of~$Q$ to this sub-lattice. This is a positive definite binary integral quadratic form. Let~$\covol(R)$ its co-volume, let~$r=4\cdot\covol(R)^2$ be its Hessian determinant, and~$\Delta(R)=-r$ its discriminant.
We assume that~$\Pi$ is primitive and denote by~$S=Q\pmod\Pi$ the quotient euclidean lattice~$\Z^3/\Pi$.  We have then
\[
\covol(Q)=\covol(R)\cdot \covol(S).
\]

Each element~$s$ of~$S$ corresponds to a coset of~$\Pi$ in~$\Z^3$, a ``slice''. The restriction of~$Q$ on the real affine (hyper)surface through~$s$, written in any affine basis of~$s$, is a quadratic polynomial~$P_s$ of the kind studied in next section: it is positive, and integer valued on~$s$. We partition integer solutions of~$Q=n$ according to the slices and obtain
\[
r(Q,n)=\sum_{s\in S} r(P_s,n), \text{ where } r(P_s,n)=\#\{\lambda\in s|P_s(\lambda)=n\}.
\]

We bound the number of slices with non zero~$r(P_s,n)$. Every slice~$s$ containing a solution of~$Q=n$ will be, as an element of~$S$, of norm at most~$n$.
The number of elements of norm at most~$n$ in the rank one euclidean lattice~$S$ is~$1+2\left\lfloor\left.\sqrt{n}\middle/{} \covol(S)\right.\right\rfloor$ elements. The first term~$1$ corresponds to the origin coset~$s=\Pi$ of norm~$0$.

Let us bound the terms. We note that the quadratic part of~$P_s$ is determined by~$R$, and hence its discriminant is~$\Delta(R)=-4\cdot \covol(R)^2$.
For the coset~$s=\Pi$ we use the original Dirichlet bound~\eqref{Dirichlet bound}
\[
r(P_0,n)=r(R,n)\leq u(\Delta(R))\cdot{\sigma_0}(n).
\]
According to the proposition~\ref{bound quadratic polynomials}, the other terms are uniformly bounded by
\[r(P_s,n)\leq u(\Delta(R))\cdot\widetilde{\sigma_0}(\Delta(R)^2n).\]

Assembling these bounds yields
\begin{equation}\label{eq3}
\frac{r(Q,n)}{u(\Delta(R))}\leq {\sigma_0}(n)+2\left\lfloor\left.\frac{\sqrt{n}}{\covol(S)}\right.\right\rfloor\cdot \widetilde{\sigma_0}(\Delta(R)^2n)
\end{equation}
As the dependency on~${\Delta(R)}$ is inexpensive, we achieve most improvement by maximising~$\covol(S)$, or equivalently minimising~$\covol(R)$.

Hermite-Rankin's constant~$\gamma_{r,n}$ of order~$r$ and dimension~$n$ is such that for any Euclidean lattice~$L$ of rank~$n$ of co-volume~$1$ there is a sub-lattice~$M$ of rank~$r$ and co-volume at most~$\gamma_{r,n}$. Scaling by~$\lambda>0$, if~$L$ has co-volume~$\lambda^n$, then~$M$ has co-volume at most~$\gamma_{r,n}\lambda^r$. Here, for~$n=3$ and~$r=2$ we take~$L=(\Z^3,Q)$, then~$\lambda^3=\abs{\det(Q)}$, and~$M$ has co-volume at most
\[
\covol(M)\leq \gamma_{2,3}\cdot\covol(Q)^{2/3}.
\]
By some duality, the constant~$\gamma_{2,3}$ equals Hermite's constant~$\gamma_{1,3}=\gamma_3=\sqrt[3]{2}$ (cf.~\cite[\S10.6, and Table~14.4.1]{Martinet}, \cite[Th.~2, and p.\ 313]{Rankin})

We take~$\Pi=M$. We detail
\[
-\Delta(R)=4\covol(R)^2\leq4{\gamma_3}^2(\covol(Q)^2)^{2/3}=4{\gamma_3}^2(H(Q)/8)^{2/3}=(2H(Q))^{2/3},
\]
\[
\frac{1}{\covol(S)}=\frac{\covol(R)}{\covol(Q)}\leq\gamma_3\covol(Q)^{-1/3}=\left(2\middle/\sqrt{H(Q)/8}\right)^{1/3}=(H(Q)/2)^{-1/6},
\]
hence~$-\Delta(R)\leq\lfloor(2H(Q))^{2/3}\rfloor$. Substituting in~\eqref{eq3} we end up with
\[
r(Q,n)\leq u(\Delta(R))\cdot\left(\sigma_0(n)+ 2\cdot \lfloor  n^{1/2}(H(Q)/2)^{-1/6}\rfloor\cdot\widetilde{\sigma_0}\left({\lfloor (2H(Q))^{2/3}\rfloor}^2\cdot n\right)\right)
\]

\subsection{Quadratic polynomials -- A Dirichlet bound} Before it was generalised by Siegel, Dirichlet gave an exact formula for the total number~$r(n)$ of representations of an integer~$n$ 
by the genus of a positive definite integral \emph{binary} quadratic form (\cite[\S11.2 (11.9)]{Iwaniec}. 
 From his formula follows
\begin{equation}\label{Dirichlet bound}
r(n)\leq u\cdot \sigma_0(n)
\end{equation}
in which~$u\in\{2;4;6\}$ is the number of automorphs, and~$n>0$.

A fortiori, the number of representations by an individual form in the genus is bounded similarly.

We are interested in similar bounds for a more general form, that of a binary quadratic polynomial
\[P(x,y)=a x^2+bxy+cy^2+dx+ey+f,\]
which we assume to take positive values at real coordinates. Its quadratic part
\[
Q(x,y)=a x^2+bxy+cy^2
\]
is then definite positive. Our assumption involving the coefficients is that~$P$ takes integer values at integer coordinates; we will see a posteriori that the coefficients of~$P$ will need to be half integers, and a couple of them integers.

Our goal is the following.
\begin{proposition}\label{bound quadratic polynomials}
For every quadratic polynomial~$P$ as above with integer values at integer coordinates, positive definite quadratic part of discriminant~$\Delta$, and with positive or zero values at real coordinates,~$r(P,n)=\#\{(x,y)\in\Z^2\mid P(x,y)=n\}$ satisfies
%
\[
r(P,n)\leq u(\Delta)\widetilde{\sigma_0}(\Delta^2n)=(n\Delta)^{o(1)},
\]
where~$u(\Delta)=\#{O(Q)}=\#\Z[(\Delta+\sqrt{\Delta})/2]^\times=
\begin{cases}
6\text{ if }\Delta=-3,\\
4\text{ if }\Delta=-4,\\
2\text{ otherwise,}
\end{cases}$ is the number of automorphs of~$Q$.
\end{proposition}

NB: the representation numbers~$r(P,n)$ are coefficients of a theta series \cite{Siegel63} \cite[\S10.3]{Iwaniec}, which satisfies transformations of a modular form of weight~$1$, for which one has Ramanujan-Petersson bounds \cite[Corollaire 4.2]{Serre-Deligne}. This yields the correct asymptotic as~$n$ diverges. But, for our purpose, we want, for modular forms arising as such theta series, uniformity results in terms of~$P$.
\subsubsection{}
We first prove our claim on the coefficients of~$P$.
\begin{lemme}[See {\cite[p.\,100/116, $\alpha)$ (2)]{Siegel63}}]\label{Integralite polynome quadratique} 
Let~$P(x,y)=a x^2+bxy+cy^2+dx+ey+f\in\R[x,y]$ be such that it takes with integer values on~$\Z^2$. 
Then~$a$, $b$, $c$, $d$, $e$ and~$f$ belong to~$\frac{1}{2}\Z$. Moreover~$b,f\in\Z$. 
\end{lemme}
\begin{proof}Evaluating at the origin we find~\[f=P(0,0)\in\Z.\]
The affine form
\[
\delta_xP=
P(x+1,y)-P(x,y)=2ax+a+by+d
\]
takes integer values. on~$\Z^2$ Evaluating at the origin we find
\[
a+d=\delta_xP(0,0)\in\Z.
\]
Taking differences again, and evaluating at the origin, yields
\begin{align*}
\delta_x\delta_xP(0,0)&=2a\in\Z,\\
\delta_y\delta_xP(0,0)&=b\in\Z.
\end{align*}
We deduce
\[
2d\in\Z-2a=\Z.
\]
Similarly, arguing with~$\delta_y$,
\[
2c,2e\in\Z.\qedhere
\]
\end{proof}
\paragraph*{Example} The triangular numbers
\[
\frac{x^2}{2}+\frac{x}{2}=\frac{x(x+1)}{2}={x+1 \choose 2}
\]
are known to be integers for integers values of~$x$. Nevertheless the coefficients aren't integers but half-integers. As for
\[{x+y \choose 2}=\frac{(x+y)(x+y-1)}{2}=\frac{x^2}{2}+xy+\frac{y^2}{2}-\frac{x}{2}-\frac{y}{2}\]
its~$b$ coefficient is integral.
\subsubsection{}
We now turn to the proof of our proposition. This is a reduction to Dirichlet bounds for the quadratic part~$Q$ of~$P$, by translating the origin at the rational point where~$P$ is minimal, and rescaling to get rid of denominators. 
The proof here does not seek optimal bounds (notably the dependency upon~$\Delta$), as Dirichlet bounds are sufficient for us.
\begin{proof}
Let~$Q(x,y)=ax^2+bxy+cy^2$ be the quadratic part of~$P(x,y)$.

It is associated to the symmetric bilinear form 
\[B((x,y),(z,t))=axz+b(xt+yz)/2+cyt=(az+bt/2)x+(bz/2+ct)y\]

As~$Q$ is assumed to be positive definite,~$(ac-b^2/4)$ is non zero.
We write 
\[2(\lambda,\mu)=(cd-be/2,-bd/2+ae)/(ac-b^2/4),\] so that we have
\[
a x^2+bxy+cy^2+dx+ey+f=Q(x,y)+2B((x,y),(\lambda,\mu))+f.
\]
We complete the square
\[
a x^2+bxy+cy^2+dx+ey+f=Q((x,y)+(\lambda,\mu))+f-Q(\lambda,\mu).
\]
The minimum of~$P$ on~$\R^2$ is~$m=f-Q(\lambda,\mu)\geq0$ at~$(x,y)+(\lambda,\mu)=0$. 

Let~$\Delta=b^2-4ac$. Using lemma~\ref{Integralite polynome quadratique} we note that~$(\lambda,\mu)\in\left(\frac{1}{\Delta}\Z\right)^2$. 

We are interested in the number of solutions~$r(P,n)$, in~$\Z^2$, of
\[
P(x,y)=n.
\]
This~$r(P,n)$ is the number of solutions, in~$\Z^2+(\lambda,\mu)$, of
\[
Q=n-f+Q(\lambda,\mu)=n-m.
\]
As we have~$\Z^2+(\lambda,\mu)\subseteq\left(\frac{1}{\Delta}\Z\right)^2$, we may bound
\begin{multline}
r(P,n)=\#\{(x,y)\in\Z^2+(\lambda,\mu)\mid Q(x,y)=n-m\}\\\leq\#\{(x,y)\in\left(\frac{1}{\Delta}\Z\right)^2\mid Q(x,y)=n-m\}.
\end{multline}
Scaling, the later becomes the set of solutions in~$\Z^2$ to
\[
Q=\Delta^2(n-f+Q(\lambda,\mu))=\Delta^2(n-m).
\]
We conclude with Dirichlet bound~\eqref{Dirichlet bound}
\[
r(P,n)\leq r(Q,\Delta^2(n-m))\leq u(\Delta)\sigma_0(\Delta^2\cdot (n-m))\leq u(\Delta)\widetilde{\sigma_0}(\Delta^2n).\qedhere
\]
\end{proof}
\appendix
\section{}
We return to the Gross lattice~$S=S_{\bar\jmath}$ (see~\ref{Gross lattice}, \cite[(12.7)]{Gross}), which is positive definite, and consider its theta series, written as a series indexed by negative discriminant numbers~$\Delta$, 
\begin{equation}\label{theta}
\theta=\theta(S;q):=\sum_{b\in S}q^{\Nmred(b)}= 1+\sum_\Delta r(\abs{\Delta},S) q^\abs{\Delta}.
\end{equation}
This is the Fourier series of a (holomorphic) modular form in the space~$M_{k}(\Gamma_1(N))$ of integer and a half weight~$k=3/2$, and of level~$N=4p$. 
The theory of half integer weight modular forms owes to~\cite{Shimura} (See~\cite{Koblitz} or~\cite{Cohen} for a treatment.) We refer to these for definitions, and the finer notion of modular form with character~$\chi$. 

See the case~$n=k=3$ of~\cite[Corollary 14.3.24]{Cohen} for the invariants~$k,N,\chi$ of~$\theta$ in terms of the invariants~$H(S),N(S)$ of~$S$. 
We compute the latter in \S\ref{Invariants Gross} below.

We remark that~$N/4=p$ is square free: we can apply the theory of~\cite{Kohnen}, and the remark of~\cite[Prop. 3.1.5]{Soma PhD} about absence of unary theta series (see also~\cite[p.\,312]{SurveySP} in terms of unicity of spinor genus.)
\subsection{Invariants of Gross' lattice}\label{Invariants Gross}
Let us recall Gross' construction of~$S=S_{\bar\jmath}$ associated with~$\bar\jmath$. 
\subsubsection{}\label{Gross lattice} We let~$B$ be a \emph{Brandt algebra} (a quaternion algebra over~$\Q$) whose \emph{reduced discriminant}~\cite[D\'{e}f. p\,58-59]{Vigneras} is the ideal~$d_B=(p)\subseteq\Z$. Then, \cite{Deuring},  the endomorphism ring of an elliptic curve~$E$ over~$\overline{\F_p}$ with~$j(E)=\bar\jmath$ is isomorphic to a maximal order~$R$ of~$B$.

We denote the canonical anti-involution~$b\mapsto \overline{b}$, the reduced norm~$\Nmred (b)=b\cdot\overline{b}\in\Z$ (the normalisation~$\Nmred(b)=-b\cdot\overline{b}$ is also found) and the reduced trace~$\Tr(b)=b+\overline{b}\in\Z$.

Then~$S$ is the sub-lattice, of rank~$3$, described by the \emph{pure quaternions} (a.k.a. traceless elements) in the order~$\Z+2R$, endowed with the restriction~$\Nmred\restriction_S:S\to\Z$ of the \emph{reduced norm form} (i.e. the reduced norm as a quadratic form).\footnote{
An equivalent quadratic lattice, up to sign, is the \emph{reduced discriminant form}~$\Delta=\Tr^2-4\Nmred$ on the quotient lattice~$R/\Z$ (not to be confused with the reduced discriminant~$d_B$ of~$B$ or~$d(R)$ of~$R$.)
}

\subsubsection{}Our aim is to recall how to compute two invariants of~$S$, for which references were not found: its \emph{Hessian determinant}~$H(S)$ and its \emph{level} \footnote{This is not to be confused with the level (as an Eichler order) considered in~\cite[III, \S5 p.\,84]{Vigneras}, though it is related.}~$N(S)$ \cite[14.3.15]{Cohen}. 
We will check 
\begin{equation}\label{Appendix1}
H(S)=+32p^2\qquad N(S)=4p,
\end{equation}
locally: firstly that~$H(S), N(S)>0$, and secondly that, for every prime~$\ell$,  we have, as ideals of~$\Z_\ell$,
\begin{equation}
H(S)\Z_\ell=32p^2\Z_\ell\text{ and }N(S)\Z_\ell=4p\Z_\ell.
\end{equation}
We rely on~\cite{Cohen} for some interpretation of these invariants and~\cite{Vigneras} for computations of related invariants. Of note, a forthcoming good reference is~\cite{Voight}.

\subsubsection{Archimedean prime}
As~$B$ is definite, the norm form on~$S$ is positive definite and its Hessian determinant~$H(S)$ is positive. The level is positive by convention.

\subsubsection{Reduced discriminant of~$R$}
As~$\overline{b}=\Tr(b)-b\in\Z+b\Z$, any order~$R=\overline{R}$ is invariant under the anti-involution. In particular it has the same dual~$R^\sharp=\{b\in B|\Tr(bR)=\Tr(b\overline{R})\subseteq \Z\}$ with respect to the two bilinear forms
\begin{equation}\label{trace form}
\beta_1(b,b')=\Tr(bb')\qquad\beta_2(b,b')=\Tr(b\overline{b'})=\Nmred(b+b')-\Nmred(b)-\Nmred(b').
\end{equation}

The \emph{reduced discriminant~$d(R)$ of~$R$}, defined in~\cite[D\'ef. p.\,24]{Vigneras}, is the ideal of~$\Z$ such that, as fractional ideals,
\begin{equation}\label{App 3}
d(R)\cdot  (\Nmred(R^\sharp))=(1)=\Z\subseteq\Q.
\end{equation}
As~$R$ is maximal,~$d(R)$ is again the ideal~$d_B=(p)$, by~\cite[III Cor. 5.3]{Vigneras}. Moreover the Gram determinant~$G(\beta_1)$ of the reduced trace form~$\beta_1$ on~$R$ satisfies~$G(\beta_1)\cdot\Z=(p)^2$ \cite[I, Lem.~4.7\,(3) (see also comments at end of the proof)]{Vigneras}. Another interpretation of this Gram determinant is, see \cite[Lem. 14.3.2]{Cohen},
\begin{equation}\label{App 5}
\abs{G(\beta_1)}=\abs{G(\beta_2)}=[R^\sharp:R].
\end{equation}
(Concerning the sign we actually have~$-G(\beta_1)=G(\beta_2)>0$.) 

Let~$S^\sharp=\{b\in \Q\cdot S|\Tr(b\overline{S})\subseteq\Z\}$.
With definition~\cite[14.3.15]{Cohen}, the level of~$S$ satisfies
\begin{equation}\label{App 4}
(N(S))\cdot (\Nmred(S^\sharp))=(1)=\Z\subseteq\Q.
\end{equation}
As for the Hessian determinant we have, by~\cite[Lem. 14.3.2]{Cohen},
\begin{equation}\label{App 6}
\abs{H(S)}=[S^\sharp:S]
\end{equation}
We will compare the two formulas~\eqref{App 3} and~\eqref{App 5} with~\eqref{App 4}  and~\eqref{App 6} at odd places. 

Also, the centre~$\Z$ has dual~
\(
\Z^\sharp=\{b\in \Q|\Tr(b\Z)\subseteq\Z\}=\frac{1}{2}\Z,
\)
has Hessian matrix~$(2)$ of determinant~$H(\Z)=2$, and has level~$N(\Z)$, such that~$(N(\Z))^{-1}=(\Nmred(\Z^\sharp))=(1/4)$.

\subsubsection{Odd Places} Let~$R_\ell=R\tens\Z_\ell$ and~$S_\ell=S\tens \Z_\ell$. Assume~$2$ is invertible in~$\Z_\ell$. We have
\[
\Z_\ell+2 R_\ell=R_\ell
\]
and a decomposition
\[
b\mapsto((b+\overline{b})/2,(b-\overline{b})/2):R_\ell\simeq \Z_\ell\oplus S_\ell
\]
which is orthogonal with respect to the bilinear forms~\eqref{trace form}.
On the center~$\Z$ these bilinear forms are equal, whereas the two are opposite when restricted to pure quaternions. 

It follows duals can be computed component-wise:~$R^\sharp\tens\Z_\ell=\Z^\sharp\tens\Z_\ell\oplus S^\sharp\tens\Z_\ell$. 

 As fractional ideals of~$\Z_\ell$, we have~$(\Nmred(\Z_\ell))=4\Z_\ell=(1)$ and  
\[
(N(S))^{-1}=(N(S))^{-1}+(1)=(\Nmred(S^\sharp))+ (\Nmred(\Z^\sharp))=(\Nmred(R^\sharp))=d(R)^{-1}.
\]
Hence~$N(S)\Z_\ell=p\Z_\ell=4p\Z_\ell$.

 The Hessian determinant of an orthogonal sum is the product of the Hessian determinants. Moreover~$H(S)=\abs{H(S)}=[S_\ell^\sharp:S_\ell]$ by~\cite[Lem. 14.3.2]{Cohen}.
 As ideals of~$\Z_\ell$, we have
\[
H(S)\Z_\ell=
([S_\ell^\sharp:S_\ell])=([\Z_\ell^\sharp:\Z][S_\ell^\sharp:S_\ell])=([R_\ell^\sharp:R_\ell])=(\abs{G(\beta_1)})=p^2\Z_\ell.
\]
It follows
\[
H(S)\Z_\ell=p^2\Z_\ell=32p^2\Z_\ell.
\]

\subsubsection{Even, finite, place}We end up with a direct computation at the prime~$\ell=2$.

We recall that the local models of~$R$ are the matrix algebras~$\End({\Z_\ell}^2)$ at primes~$\ell\neq p$, and given by~\cite[II \S1 Th.\,1.3]{Vigneras} at~$p$. In particular, at~$\ell=p=2$, it is the local model of the order of Hurwitz quaternions. As both~$H(S)$ and~$N(S)$ can be computed locally, it will suffice to consider these two examples.

We compute directly with these two examples: the Hurwitz quaternions with~$d_B=(2)$ and the matrix algebra~$\End(\Z^2)$ with~$d_B=\Z$. Recall the reduced norm forms
\[t+ix+jy+zk\mapsto t^2+x^2+y^2+z^2\text{ and~}\left(\begin{smallmatrix}a&b\\c&d\end{smallmatrix}\right)\mapsto ad-bc.\] The Hessian discriminant and the level can be computed on the Hessian matrix of these forms in any explicit basis.  A basis of~$S$ is~$2i,2j,i+j+k$ and~$\left(\begin{smallmatrix}0&2\\0&0\end{smallmatrix}\right),\left(\begin{smallmatrix}1&0\\0&-1\end{smallmatrix}\right),\left(\begin{smallmatrix}0&0\\2&0\end{smallmatrix}\right)$. Corresponding quadratic forms~$(2x+z)^2+(2y+z)^2+2z^2$ and~$y^2-4(xz)$ with Hessian matrices
\[
\left(
\begin{smallmatrix}
8&&4\\
&8&4\\
4&4&6
\end{smallmatrix}
\right)
=8\cdot
\left(
\begin{smallmatrix}
4&1&-2\\
1&4&-2\\
-2&-2&2
\end{smallmatrix}
\right)^{-1}\quad\text{and}\quad
\left(
\begin{smallmatrix}
2&&\\
&&4\\
&4&
\end{smallmatrix}
\right)
=4\cdot
\left(
\begin{smallmatrix}
2&&\\
&&1\\
&1&
\end{smallmatrix}
\right)^{-1}.
\]
We gather the Hessian determinants~$H(S)=128=+32\cdot p^2$ and\footnote{We have the minus sign as the matrix algebra is indefinite.}~$H(S)=-32=-32\cdot 1^2$, and the levels~$N(S)=8=4\cdot p$ and~$N(S)=4\cdot 1$.

\subsubsection{Evidence}We end with some numerical evidence. (One can verify in the documentation of the commands below, that these deal with Gross' lattice. N.B.: The invariants which are local in nature do not depend on the choice of a particular maximal order.) We picked~$p=163$.
\begin{quote}{\small
\begin{verbatim}
sage: version()
    'SageMath version 8.2, Release Date: 2018-05-05'
sage: Q=BrandtModule(163).maximal_order().ternary_quadratic_form()
sage: Q=QuadraticForm(ZZ,Q.Hessian_matrix())
sage: Q.det().factor()
    2^5 * 163^2
sage: Q.level().factor()
    2^2 * 163
\end{verbatim}}
\end{quote}

\subsection{General discriminant numbers}Here we obtain bounds similar to~\eqref{Duke ternary}, but including fundamental discriminants which are even, and then discriminants which are not fundamental. The references we use for the odd fundamental discriminants (square-free discriminants) might work for all fundamental discriminants  instead of merely square-free numbers, but the part of the proofs which should imply this do not provide enough details to ensure the uniformity we need, and furthermore some limitations on the conductor are present. 
\subsubsection{}
We consider the operator~$T_4$ acting on~$ M_k(\Gamma_1(N))$, which is given on~$q$-expansion by
\[
\sum_{n\geq0} a_n q^n\mapsto\sum_{n\geq0} a_{4n}q^n,
\]
and on half-period ratio coordinate~$\tau$ by
\[
T_4(f)(\tau)=\sum_{0\leq b<4}f\left(\frac{\tau+b}{4}\right).
\]
This is also the square~${U_2}^2$ of the ``U operator'' often denoted~$U_2$. A construction of~$T_4$, in the half-integral weight context, and that it preserves weight, level (and character), is explained in \cite[IV \S3 Problems 3, 5]{Koblitz}, see also~\cite[Prop. 1.3-5]{Shimura} which includes non square levels,  and that it preserves the cuspidal subspace~$S_k(\Gamma_1(N))$ is used in \cite{Soma}, or see the original reference~\cite[Prop. 1.3]{Shimura}.

\subsubsection{}

We let~$f=\sum_{n\geq0}a_n(f)$ be the cuspidal part of our series~$\theta$ from~\eqref{theta}, which we decompose, as in~\cite[(1.1)]{Kane},
\[
f=\sum_{i\in I} b_i\cdot g_i
\]
as a linear combination of orthogonal simultaneous eigenforms~$g_i$ for the anemic Hecke operators. We choose the~$g_i$ to be analytically normalised, that is such that the~$g_i$ have Petersson norm~$1$. \emph{The normalisation we use for the Petersson norm~$\Nm{\placeholder}$ is $\Nm{\placeholder}_{M_k(\widetilde{\Gamma_1(N)})}$ from~\eqref{Petersson norm}, with~$\widetilde{\Gamma_1(N)}$ as denoted in~\cite[IV \S3, p.\,182]{Koblitz}, which is the one used by~\cite{Iwaniec Fourier,Duke hyperbolic,Duke ternary}.} By orthogonality of the~$g_i$,
\begin{equation}\label{orthogonality}
\Nm{f}^2=\sum_{i\in I} \abs{b_i}^2.
\end{equation}

\subsubsection{Petersson norm bound}
We now use\footnote{After correcting the~$\eps=2+2\eps$ and the~$\abs{a(n)}\leq n^{1/2}D^{3/2}(nD)^\eps\neq N^{1/2}D^{3/2}(nD)^\eps$.} \cite[end of \S3]{Duke ternary} to bound for the Petersson norm~$\Nm{f}$. Here~$D=32p^2$ and~$N=4p$. 
There are constants~$\eps',\eps''=8\eps',C_1(\eps'),C_2(\eps'')=2^{23+28\eps'}C_1(\eps')>0$ such that
\begin{equation}\label{Duke bound}
\Nm{f}^2\leq C_1(\eps')\cdot N^4D^3(ND)^{4\eps'}=C_2(\eps'')\cdot p^{10+\eps''}.
\end{equation}
This bound is uniform in~$S$, in particular in~$p$.  In our case, we might improve the polynomial exponent with the sharper bound~$M(G)\approx\frac{p-1}{24}$ instead of~$M(G)\ll D^2$ in~\cite[after (8)]{Duke ternary}. The polynomial growth is all that matters for our main problem.
\subsubsection{Including fundamental discriminants which are even}

Let~$\phi=\sum_{n\geq0}a_n(\phi)\in S_{k}(\Gamma_1(N))$ be one of the~$g_i$ or of the~$T_4(g_i)$. In the latter case we
have a variant of Hecke bound from lemma~\ref{Hecke bound}, which, with~$C_4=4^{3/4}\geq 1$,  gives us
\(
\Nm{\phi}=\Nm{T_4(g_i)}\leq C_4\Nm{g_i}=C_4
\). In any case
\[
\Nm{\phi}\leq C_4.
\]

We can apply  Duke's results~\cite{Duke hyperbolic} to~$\phi$, in the form given by~\cite[Lemma 2]{DSP}, in the case of square free~$n=t=t'$: with~$n_0=1$ and~$S=\{\}$, in their notations. For square-free~$n$ we do not need the hypothesis on the Shimura lift, as is seen in their proof; yet this hypothesis actually holds in our case: in~$S_{k}(4p)$ the subspace spanned by unary th{e}ta series \cite[\S2]{Soma IJNT}, is zero, \cite[Prop. 3.1.5 (Kohnen)]{Soma PhD}. 

The results invoked  above give the domination, for some constant~$C_3(\eps)$ depending only on~$\eps>0$ (and not on our quadratic module~$S$, or the choice of~$\phi$), for all square-free~$n$,
\begin{equation}\label{DukeIwaniec}
\abs{a_n(\phi)}\leq C_3(\eps)\cdot \Nm{\phi} \cdot n^{1/2-\gamma+\eps}\leq C_4\cdot C_3(\eps)\cdot n^{1/2-\gamma+\eps}
\end{equation}
with~$0<\eps<\gamma=-1/28$. The Petersson norm is not defined, hence not explicitly normalised, in the reference~\cite{DSP}. From the proof it seems to be the normalisation of~\cite{Iwaniec Fourier,Duke hyperbolic,Duke ternary}, or at least the proof seems to work with such normalisation. Most other used normalisations \emph{involve a polynomial dependency of~$C(\eps)$ on the level~$N$}, which is acceptable for our application in this article.

We know focus on one of the~$g_i$, and pick~$-\Delta$ for a negative fundamental discriminant~$\Delta$. 
If~$n=-\Delta$ is square-free, we apply~\eqref{DukeIwaniec} for~$\phi=g_i$, otherwise~$n=-\Delta/4$ is square-free and we apply~\eqref{DukeIwaniec} for~$\phi=T_4(g_i)$. 
We obtain
\[\abs{a_{-\Delta}(g_i)}=
\abs{a_{n}(\phi)}\leq C_4\cdot C_3(\eps)\cdot n^{1/2-\gamma+\eps}\leq C_4\cdot C_3(\eps)\cdot \abs{\Delta}^{1/2-\gamma+\eps},\\
\]
We omitted the possible factor~$\abs{1/4}^{1/2-\gamma+\eps}\leq 1$, as~$\gamma=1/28\leq1/2$. (One could improve~$C_4$ accordingly.)
\subsubsection{Including nontrivial conductors} We now introduce the conductor~$F$ of a negative discriminant number~$\Delta\cdot F^2$. The following is a classical argument involving Shimura lifts~\cite{Shimura} and Deligne's Ramanujan-Petersson~\cite[9.3.2]{Cohen} bounds~\cite[Th 5.6]{Deligne71}, \cite{DeligneWeil1} (in weight $2$, so ultimately on Hasse-Weil bound from Weil's Riemann hypothesis. See for instance~\cite{Lub App} for an account.)

We rely on a version of this argument from \cite{Kane}, which is tailored to our case. Note that contrary to most part of the reference~\cite{Kane}, this does not rely on the conjectural Generalised Riemann Hypothesis, only the proved Riemann Hypothesis for curves over finite fields. We apply \cite[Lem. 4.1]{Kane} and deduce, for any negative discriminant~$\Delta\cdot F^2$ of conductor~$F\geq 1$,
\[
\abs{a_{-\Delta\cdot F^2}(g_i)}\leq a_{-\Delta}(g_i) \cdot \sigma_0(F)^2\cdot F^{1/2}.
\]
For~$\eps>0$, there is a constant~$C_5(\eps)$ such that~$\sigma_0(F)\leq C_5(\eps)\cdot F^{\eps}$. We find
\begin{equation}\label{Duke etendu}
\abs{a_{-\Delta\cdot F^2}(g_i)}\leq C_5(\eps)\cdot{} C_4\cdot C_3(\eps)\cdot\abs{\Delta}^{1/2-\gamma+\eps}\cdot  (F^{2})^{1/4+\eps}\leq C_5(\eps)\cdot C_4\cdot C_3(\eps)\cdot\abs{\Delta F^2}^{1/2-\gamma+\eps}.
\end{equation}

We now go back to~$f$ and get, with~\eqref{orthogonality},
\[
\abs{a_{-\Delta\cdot F^2}(f)}^2=\sum \abs{b_i}^2\cdot\abs{a_n(g_i)}^2\leq\Nm{f}\cdot C_5(\eps)\cdot C_4\cdot C_3(\eps)\cdot\abs{\Delta F^2}^{1/2-\gamma+\eps}.
\]
We plug~\eqref{Duke bound} and conclude, for~$\eps,\eps''>0$, for any discriminant number~$\Delta F^2$,
\begin{equation}\label{final bound}
\abs{a_{-\Delta\cdot F^2}(f)}^2\leq C_2(\eps'')\cdot C_5(\eps)\cdot C_4\cdot C_3(\eps)\cdot\abs{\Delta F^2}^{1/2-\gamma+\eps}\cdot p^{10+\eps''}.
\end{equation}

\subsection{A Hecke bound} We use a variant of Hecke's bound (see for instance~\cite[Lem. 3.62, p.\,90]{ShimuraIntro}) valid for the non necessarily diagonalisable operator~$U_4$, encompassing integer and a half weights, and featuring uniform constants with respect to variations of level. (A polynomial dependence on the level would be enough for our concern.)

We will only need~$k=3/2$ and~$n=4$, or even~$n=2$. We consider, on Fourier series, for~$n\geq 0$, the operator
\[
U_n:\sum_{n\geq0}a_mq^m\mapsto\sum_{m\geq0} a_{nm}q^m.
\]
(See~\cite[III \S5 Prop.\,37]{Koblitz} for the link between~$U_\ell$ and~$T_\ell$ for a prime~$\ell$.) On non necessarily holomorphic functions, we let~$U_n$ act more generally via
\[
U_n(f):\tau\mapsto\frac{1}{n}\sum_{0\leq b<n}f((\tau+b)/n).
\]

\subsubsection*{Half-integer weight setting}
In half-integer weight, we work: firstly with a group~$G$ which is an extension (related to the ``metaplectic group'') of the group~$GL(2,\R)^+=\{g\in GL(2,\R)|\det(g)>0\}$  which acts, through the quotient~$P:G\to GL(2,\R)^+\to PGL(2,\R)^+=P(G)$, on the Poincar\'{e} half-plane~$H=\{\tau=x+iy|x\in\R,y\in\R_{>0}\}$; secondly with, for every~$k\in\frac{1}{2}\Z$,  a right action~$\placeholder|_k\placeholder$ on functions~$\phi:H\to\C$ which is  such that
\begin{equation}\label{invariance k k}
\left(\abs{\phi|_kg}^2y^k\right)(\tau)=\left(\abs{\phi}^2 y^k\right)(g\cdot \tau),
\end{equation} and 
that this action preserves the subspace of measurable (resp. holomorphic) functions. (This follows in practice from the explicit form of the automorphy factor, which we avoided talking about).

\subsubsection*{Petersson ``norms''}
We consider the vector space~$M_k(\Gamma)$ of functions, assumed measurable (and holomorphic if one wishes), which are invariant  under a subgroup~$\Gamma\leq G$ such that~$P(\Gamma)$ is a discrete subgroup in~$PGL(2,\R)^+$.
Recall that the measure~$\mu=dxdy/y^2$ is invariant under~$PGL(2,\R)^+$. The quotient measure of~$\abs{\phi}^2 y^k \mu$, by the counting Haar measure~$\mu_{P(\Gamma)}:E\subseteq P(\Gamma)\mapsto\abs{E}$ on~$P(\Gamma)$, is a well defined positive or zero measure on~$P(\Gamma)\backslash H$ and the Petersson ``norm'' (it need not be finite) of~$\phi$ in~$M_k(\Gamma)$
\begin{equation}\label{Petersson norm}
\Nm{\phi}_{M_k(\Gamma)}=\sqrt{
\int_{P(\Gamma)\backslash H} \left.\mu_{P(\Gamma)}\middle\backslash\abs{\phi}^2 y^k \mu\right.}=\sqrt{\int_F\abs{\phi}^2 y^k \mu}\quad\text{ ($=\Nm{\phi}_{L^2(F,y^k\mu)}$ if finite)}
\end{equation}
is well defined in~$\R_{\geq0}\cup\{+\infty\}$, where~$F$ stands for a measurable fundamental domain for~$P(\Gamma)$ acting on~$H$. See~\cite[I VII \S10, esp. Th. 4]{BBK I7} about quotient measure and fundamental domains.

\subsubsection{Statement}

Our aim is the following.
\begin{lemme}[Hecke bound]\label{Hecke bound}Assume the level~$\Gamma$ is a subgroup~$\Gamma\leq G$ such that~$P(\Gamma)$ is a lattice of~$PGL(2,\R)^+$ contained in~$PGL(2,\Q)$. 
Assume that~$f$ and~$U_n(f)$ both belong to~$M_k(\Gamma)$. 

Then we have, in~$\R_{\geq0}\cup\{+\infty\}$,
\begin{equation}\label{Hecke bound n}
\Nm{U_n(f)}_{M_k(\Gamma)}\leq n^{k/2}\cdot \Nm{f}_{M_k(\Gamma)} 
\end{equation}
In particular, if~$f$ is an eigenform of~$U_n$ with eigenvalue~$\lambda_n$ and finite (non zero) Petersson norm, then
\begin{equation}\label{Eigenvalue bound}
\abs{\lambda_n}\leq n^{k/2}.
\end{equation}
\end{lemme}
N.B. In usual settings, say with~$G$ as in~\cite{Koblitz}, the hypothesis~$U_n(f)\in M_k(\Gamma)$ can be satisfied by passing to a sufficiently small finite index subgroup of~$\Gamma$. Even more, this hypothesis even becomes transparent if one normalises the Petersson ``norm'' so that it become insensitive to passing to a finite index subgroup.

N.B. Our method works similarly for the classical operators~$T_n$, by a suitable choice of family~$A\subseteq PGL(2,\Q)^+$ in the proof.



\subsubsection{}
We start with some formal properties.
\begin{lemme} For~$g$ in~$G$ and~$\phi$ in~$M_k(\Gamma)$, the function~$\phi|_kg$ belongs to~$M_k(g^{-1}\Gamma g)$.

 We have
\begin{equation}\label{isometrie}
\Nm{\phi|_k g}_{M_k(g^{-1}\Gamma g)}=\Nm{\phi}_{M_k(\Gamma)}.
\end{equation}

 If~$\Gamma'$ is a subgroup of~$\Gamma$ then~$M_k(\Gamma)\subseteq M_k(\Gamma')$ and
\begin{equation}\label{sous groupe}
\Nm{\phi}_{M_k(\Gamma')}=\sqrt{[P(\Gamma):P(\Gamma')]}\cdot\Nm{\phi}_{M_k(\Gamma)}.
\end{equation}
\end{lemme}
\begin{proof}
The first assertion contains two informations: the invariance~$\phi|_k g$, which follows from the property of group actions; its measurability (resp. holomorphy), which was assumed in the setting.

The second assertion  (compare~\cite[III Prop. 42]{Koblitz} for integral weight) follows from~\eqref{invariance k k}, which gives us more: the bijection
\[\Gamma\tau\mapsto P(g)^{-1}\Gamma\tau=P(g^{-1}\Gamma g) P(g)^{-1} \tau:P(\Gamma)\backslash H\to P(g^{-1}\Gamma g)\backslash H\]
conjugates the measure~$\mu_{P(\Gamma)}\backslash\left(\abs{\phi}^2y^k\mu\right)$ with~$\mu_{P(g^{-1}\Gamma g)}\backslash\left(\abs{\phi|_kg}^2y^k\mu\right)$. We integrate and get~\eqref{isometrie}.

As for the third assertion: we invoke a system~$R$ of representatives~$P(\Gamma)=\coprod_{r\in R} P(\Gamma')\cdot r$ and a measurable fundamental domain~$F$ for~$P(\Gamma)$. Then~$F'=\coprod_{r\in R}r^{-1}F$ is a measurable fundamental domain of~$P(\Gamma')$. We integrate
\[\int_{F'}\abs{\phi}^2 y^k \mu=\sum_{r\in R} \int_{r^{-1}F}\abs{\phi}^2 y^k \mu=\sum_{r\in R} \int_{F}\abs{\phi}^2 y^k \mu=\abs{R}\int_{F}\abs{\phi}^2 y^k=
[P(\Gamma):P(\Gamma')]\int_{F}\abs{\phi}^2 y^k,\]
where the second equality follows form the base change formula~\eqref{invariance k k} or from~\cite[I VII \S10, Cor. to Th. 4]{BBK I7}. We conclude by taking square root of leftmost and rightmost members.
\end{proof}
\subsubsection{}

Pick~$\alpha\in PGL(2,\R)^+$ corresponding to an affine transformation~$\alpha\cdot \tau=\lambda_\alpha\tau+\mu_\alpha$ and let~$g_\alpha$ be a lift of~$\alpha$ in~$G$: one has~$P(g_\alpha)=\alpha$. Then, from~\eqref{invariance k k}, and~$y(\alpha\cdot\tau)=\lambda_\alpha\cdot y(\tau)\in\C^\times$, follows
\begin{equation}
\abs{\phi}^2(\alpha\cdot\tau)\cdot \lambda_\alpha^k=\abs{\phi|_kg_\alpha}^2(\tau).
\end{equation}
We apply this relation to~$\phi=f\in M_k(\Gamma)$ and denote~$f_\alpha:\tau\mapsto f(\alpha\cdot \tau)$. We note that the integrand~$\abs{f_\alpha}^2y^k\mu=\lambda_\alpha^k\abs{f|_kg_\alpha}^2y^k\mu$ is invariant under~$P(\Gamma)_\alpha=\alpha^{-1}P(\Gamma)\alpha$ and a fortiori under any subgroup~$\Pi$. We  can integrate, then use~\eqref{sous groupe} for~$\phi=f|_kg_\alpha$ and~$\Gamma'=\stackrel{-1}{P}(\Pi)\cap\Gamma_\alpha$, and use~\eqref{isometrie} for~$\phi=f$,
\begin{equation}\label{computation 1}
\begin{split}
\lambda_\alpha^k\int\!\!\!\!\!\mu_{\Pi}\backslash\abs{f_\alpha}^2y^k\mu
=
\int\!\!\!\!\! \mu_{\Pi}\backslash\abs{f|_kg_\alpha}^2y^k\mu&=
\int\!\!\!\!\! \mu_{P(\Gamma')}\backslash\abs{f|_kg_\alpha}^2y^k\mu\\&=\Nm{f|_kg_\alpha}^2_{M_k(\Gamma')}\\
&=
[P(\Gamma)_\alpha:P']\cdot\Nm{f|_kg_\alpha}^2_{M_k({g_\alpha}^{-1}\Gamma{g_\alpha})}\\
&=
[P(\Gamma)_\alpha:P']\cdot\Nm{f}^2_{M_k(\Gamma)}.
\end{split}
\end{equation}

\subsubsection{}
 We now consider a finite family~$A$ of such elements~$\alpha$ of~$PGL(2,\Q)^+$, and we introduce~$U_A(f)=\sum_{\alpha\in A}f_\alpha$, which we assume to belong to~$M_k(\Gamma)$. We will prove, under the assumptions of the lemma~\ref{Hecke bound},
\begin{equation}\label{Hecke bound A}
\Nm{U_A(f)}_{M_k(\Gamma)}\leq 
\left(\sum_{\alpha\in A}\lambda_\alpha^{-k/2}\right)\Nm{f}_{M_k(\Gamma)}.
\end{equation}
Specialising to the case~$A=(\pm\left(\begin{smallmatrix}1&b\\0&n\end{smallmatrix}\right))_{0\leq b<n}$ and~$\lambda_\alpha=1/n$ we recover~$U_n=\frac{1}{n}U_A$ and~\eqref{Hecke bound A} gives~\eqref{Hecke bound n}. 
\begin{proof}
Each~$\alpha$ belongs to the commensurator of~$P(\Gamma)\leq PGL(2,\Q)^+$: there is hence a common subgroup~$\Pi$ of finite index in~$P(\Gamma)$ and in each of the~$P(\Gamma)_\alpha$. We may write, using~\eqref{sous groupe} for~$\phi=U_A(f)$ and~$\Gamma'=\stackrel{-1}{P}(\Pi)\cap\Gamma$,
\begin{equation}\label{computation 2}
\begin{split}
[P(\Gamma):P(\Gamma')]\cdot \Nm{U_A(f)}^2_{M_k(\Gamma)}
=
\Nm{U_A(f)}^2_{M_k(\Gamma')}&=
\int\!\!\!\!\!\mu_{P(\Gamma')}\backslash\abs{U_\alpha(f)}^2y^k\mu=
\int\!\!\!\!\!\mu_{\Pi}\backslash\abs{\sum_{\alpha\in A}f_\alpha}^2y^k\mu.
\end{split}
\end{equation}
In order to prove~\eqref{Hecke bound A} we may assume~$\Nm{f}_{M_k(\Gamma')}<+\infty$. Let~$F$ be a fundamental domain for~$\Pi$. We then have, in~$L^2(F,y^k\mu)$,  the triangle inequality 
\begin{equation}\label{computation 3}
\sqrt{\int\!\!\!\!\!\mu_{\Pi}\backslash\abs{\sum_{\alpha\in A}f_\alpha}^2y^k\mu}
=
\sqrt{\int_{\!\!\!F}\!\abs{\sum_{\alpha\in A}f_\alpha}^2y^k\mu}
\leq
\sum_{\alpha\in A}
\sqrt{\int_{\!\!\!F}\!\!\abs{f_\alpha}^2y^k\mu}
=
\sum_{\alpha\in A}
\sqrt{\int\!\!\!\!\!\mu_{\Pi}\backslash\abs{f_\alpha}^2y^k\mu}
\end{equation}
where we duly noted that each~$\abs{f_\alpha}^2y^k\mu$ is~$\Pi$ invariant and of finite integral as seen with~\eqref{computation 1}.
N.B.~We did not use a priori that the sums~$\sum_{\alpha\in I}f_\alpha$, with~$I\subseteq A$, are in some~$M_k(\Gamma'')$.

We plug together~\eqref{computation 2}, \eqref{computation 3} with~\eqref{computation 1}, yielding
\[
\sqrt{[P(\Gamma):\Pi]}\cdot \Nm{U_A(f)}_{M_k(\Gamma)}
\leq
\sum_{\alpha\in A}
\sqrt{[P(\Gamma_\alpha):\Pi]}\cdot\lambda^{-k/2}\Nm{f}_{M_k(\Gamma)}.
\]
We can conclude with the claim that~$
[P(\Gamma):\Pi]=
[P(\Gamma)_\alpha:\Pi]$, which are non zero, and simplifying the leading factors. This claim follows for instance, for~$\phi=1$ and~$k=0$ from~\eqref{isometrie} and~\eqref{sous groupe}, and that~$\Nm{1}_{M_0(\Gamma)}\in\R_{>0}$, which is from the assumption that~$P(\Gamma)$ is a lattice. 
\end{proof}

\end{document}